\documentclass[10pt, leqno]{amsart}
\usepackage{amssymb}
\usepackage{amsmath,amscd}
\usepackage[initials,nobysame,alphabetic]{amsrefs}
\usepackage{amsmath, amssymb}
\usepackage{amsfonts}
\usepackage{mathrsfs}
\usepackage{mathpazo}
\usepackage{mathtools}
\usepackage[arrow,matrix,curve,cmtip,ps]{xy}
\usepackage[utf8x]{inputenc}
\usepackage{amsthm}

\usepackage{float}
\usepackage{hyperref}
\usepackage{graphics}
\usepackage{graphicx}
\usepackage{verbatim}
\usepackage{multirow}
\usepackage{tikz}
\usepackage{enumerate}
\allowdisplaybreaks

\newtheorem{theorem}{Theorem}

\newtheorem{definition}[theorem]{Definition}

\newtheorem{proposition}[theorem]{Proposition}
\newtheorem{remark}[theorem]{Remark}

\newcommand{\E}{\mathbb{E}}
\newcommand{\R}{\mathbb{R}}
\newcommand{\Ff}{\mathbb{F}}
\newcommand{\Lf}{\mathbb{L}}

\newcommand{\Ec}{\mathcal{E}}

\newcommand{\F}{\mathcal{F}}

\newcommand{\B}{\mathcal{B}}
\newcommand{\cS}{\mathcal{S}}

\newcommand{\hH}{\mathcal{H}}
\newcommand{\cK}{\mathcal{K}}

\newcommand{\Pf}{\mathbb{P}}

\def\as{\hbox{\rm -a.s.{ }}}

\def\d{\delta}             

\begin{document}

\title[Reflected Backward Stochastic Volterra Integral Equations]{Reflected Backward Stochastic Volterra Integral Equations and related time-inconsistent optimal stopping problems}
\author{Nacira Agram and Boualem Djehiche}

\address{Department of Mathematics, Linnaeus University (LNU), V{\"a}xj{\"o}, Sweden}
\email{nacira.agram@lnu.se}

\address{Department of Mathematics \\ KTH Royal Institute of Technology \\ 100 44, Stockholm \\ Sweden}
\email{boualem@kth.se}

\date{ \today}



\begin{abstract}
We study  solutions of a class of  one-dimensional continuous reflected backward stochastic Volterra integral equations driven by Brownian motion, where the reflection keeps the solution above a given stochastic process (lower obstacle). We prove existence and uniqueness by a fixed point argument and we derive a comparison result. Moreover, we show how the solution of our problem is related to a time-inconsistent optimal stopping problem and derive an optimal strategy.
\end{abstract}

\maketitle

\tableofcontents

\section{Introduction}

In recent years backward stochastic Volterra integral equations (BSVIEs) have attracted a lot of interest due to their modeling potential in problems related to  time-preferences of decision makers, asset allocation and risk management in mathematical finance among other fields, for which the related optimal control problems are time-inconsistent meaning that  the associate value-function does not satisfy the dynamic programming principle.

Lin \cite{lin} was first to introduce and study a class of BSVIEs driven by Brownian motion. They can be described as follows: For a given Lipschitz driver $f$ and a square integrable terminal condition $\xi$ solving a BSVIE consists in finding an adapted process $(Y,Z)$ satisfies the equation
\[
 Y(t)=\xi+\int_{t}^{T}f(t,s,Y(s),Z(t,s))ds-\int_{t}^{T} Z(t,s)dW(s).
\]
In a series of papers, Yong \cite{Yong1, Yong2, Yong3} systematically studied a general class, including the M-solution, of BSVIEs of the form:
\[
 Y(t)=\xi(t)+\int_{t}^{T}f(t,s,Y(s),Z(t,s),Z(s,t))ds-\int_{t}^{T} Z(t,s)dW(s).
\]
This was followed by a quite extensive list of papers including (in alphabetic order) Djordjevi{\'c} and Jankovi{\'c} \cite{DJ1,DJ2}, Shi, Wang and Yong \cite{swy}, Wang and Yong \cite{wy}, Wang and Zhang \cite{WZ}, Hu and {\O}ksendal \cite{HO}, Wang and Yong \cite{wy}, Popier \cite{P} and Wang, Yong and Zhang \cite{wyz}, to mention a few.

In problems related to mathematical finance, the BSVIE is satisfied by the value function of  time-inconsistent optimal control and equilibrium problems related to stochastic differential utility such as in  Di Persio \cite{dip}, dynamic risk measures such as in Yong \cite{Yong2}, Wang and Yong \cite{wy} and Agram \cite{Ag} and  dynamic capital allocations such as in  Kromer and Overbeck \cite{KO}. The BSVIE is also satisfied by the adjoint process of a related stochastic maximum principle related to equilibrium recursive utility and equilibrium dynamic risk measures, such as in Djehiche and Huang \cite{DH}, Wang, Sun and Yong \cite{wsy}  among many more papers (see \cite{wyz} for further references).

The purpose of the present work is to prove, under certain conditions,  existence and uniqueness of solutions $(Y,Z,K)$ of reflected BSVIEs of the form
\[
Y(t)=\xi(t)+\int_{t}^{T}f(t,s,Y(s),Z(t,s))ds+\int_t^TK(t,ds)-\int_{t}^{T} Z(t,s)dW(s),
\]
where $W$ is Brownian motion. The equation is called reflected because the first component $Y$ of the solution is constrained to be greater than or equal to a given obstacle process $L$. Standard reflected backward stochastic differential equations (BSDEs) were initially studied in the seminal paper by El Karoui {\it et al.} \cite{elkaroui}. In this present paper we investigate to what extent  the known results on standard reflected BSDEs carry over to the Volterra type of reflected BSVIEs, especially which type of Skorohod flatness condition suites this class of processes.  We also show that the solutions to this class of reflected BSVIEs are in fact the value functions of optimal stopping problems related to time-inconsistent stochastic utility functions and we derive an optimal stopping strategy. 

The content of our paper is as follows. After some preliminaries in Section \ref{F}, we give a formulation of the class of continuous reflected BSVIEs with lower obstacle. In Section \ref{EU} we derive existence, uniqueness and comparison results. Finally, in Section \ref{OS}, we give an application to time-inconsistent optimal stopping problems.

\subsection*{Notation and preliminaries}\label{prem}
Let $(\Omega,\F,\Pf)$ be a complete probability space on which is defined a standard one-dimensional Brownian motion $W=(W(t))_{0\le t\leq T}$. We denote by $\mathbb{F}:=(\mathcal{F}_{t})_{0\le t\leq T}$ its natural filtration augmented by all the $\Pf$-null sets in $\F$.  

Let $\mathbb{J}$ be a nonempty set and consider two processes $X=(X_s, \,s\in \mathbb{J})$ and $Y=(Y_s,\, s\in \mathbb{J})$ defined on the same probability space $(\Omega,\F,\Pf)$.   We say that $X$ is a {\it modification} of $Y$ if  for each $s\in \mathbb{J}$ , 
$$
\Pf( \omega\in \Omega: \, \, X_s(\omega)=Y_s(\omega))=1.
$$
We say that $X$ and $Y$ are {\it indistinguishable} if 
$$
\Pf( \omega\in \Omega: \, \, X_s(\omega)=Y_s(\omega);\,\,\forall s\in \mathbb{J})=1.
$$
If both $X$ and $Y$ are a.s. continuous then if  $X$ is a modification of $Y$, they are indistinguishable.

Finally, we  let $\B(G)$ be the Borel $\sigma$-field of the metric space $G$. In the sequel, $C>0$ represents a generic constant which can be different from line to line.

\medskip In this paper we only consider reflected BSVIEs driven by a one-dimensional Brownian motion. Extension to higher dimensions being straightforward.

\medskip
We define the following spaces for the solution.

\medskip
\begin{itemize}
\item $\cS^{2}$ is the set of ${\mathbb{R}}$-valued $\mathbb{F}
$-adapted processes $(Y(u))_{u\in[0,T]}$ such that
\[
{\Vert Y\Vert}_{\mathcal{S}^{2}}^{2}:=\E[\sup_{t\in\lbrack
0,T]}|Y(t)|^{2}]~<~\infty\;.
\]

\item $\hH^2$ is the space of $\Ff$-adapted processes $(v(u))_{u\in[0,T]}$ such that $$
\|v\|^2_{\hH^2}:=\E\left[\int_0^T|v(s)|^2ds\right]<\infty.
$$

\item $\Lf^{2}$ is the set of ${\R}$-valued  processes $(Z(t,s))_{(t,s)\in [0,T]\times [0,T]}$ such that
for almost all $t\in[0,T]\,\,\, Z(t,\cdot)\in\hH^2$ and satisfy
\[
\Vert Z\Vert_{\Lf^{2}}^{2}:=\E\left[\int_{0}^{T}\int_{t}^{T}|Z(t,s)|^{2}dsdt\right]<~\infty\;.
\]
\item  $\cK^2$ is the space of processes $K$  which satisfy
\begin{itemize}
\item[$\bullet$] for each $t\in [0,T],\, u\mapsto K(t,u)$ is an $\Ff$-adapted and increasing process with $K(t, 0)=0$;
\item [$\bullet$] $(t,u)\mapsto K(t,u)$ is continuous and  $K(\cdot,T)\in\hH^2$.
\end{itemize}
\end{itemize}

\medskip The spaces $(\cS^2, {\Vert \cdot\Vert}_{\mathcal{S}^{2}}^{2}), (\hH^2, {\Vert \cdot\Vert}_{\hH^{2}})$ and $(\Lf^2, \|\cdot\|_{\Lf^2})$ are Hilbert spaces. 

\section{Formulation of the problem}\label{F}
We investigate existence of a unique triple $(Y,Z,K)$ of processes taking values in $\R\times\R\times\R^+$ which satisfy the following reflected BSVIE with one obstacle associated with $(f,\xi,L)$:
\begin{equation}\label{rbsde-1}
Y(t)=\xi(t)+\int_{t}^{T}f(t,s,Y(s),Z(t,s))ds+\int_t^TK(t,ds)-\int_{t}^{T} Z(t,s)dW(s),
\end{equation}
such that
\begin{itemize}
\item[(a)] $Y\in\hH^2,\,t\mapsto Y(t)$ is continuous and $Z\in\Lf^2$; 
\item[(b)] $Y(t)\ge L(t)\quad \Pf\as,\quad 0\le t\le T$;
\item[(c)] The process $K$ enjoys the following properties: 
\begin{itemize}
\item[(c1)] $K\in\cK^2$ i.e. for each $t\in [0,T],\, u\mapsto K(t,u)$ is an increasing process with $K(t, 0)=0$, $(t,u)\mapsto K(t,u)$ is continuous and  $K(\cdot,T)\in\hH^2$; 
\item[(c2)]  The Skorohod flatness condition holds:  
for each $0\le \alpha<\beta \le T$,
$$\begin{array}{lll}
K(u,\alpha)=K(u,\beta) \quad\text{whenever}\,\, Y(u)>L(u) \,\, \text{for each}\,\,\, u\in[\alpha, \beta] \quad \Pf\as  
\end{array}
$$
\end{itemize}
\end{itemize}

\medskip
\begin{definition}\label{rbsvie} We say that the triple of progressively measurable processes $(Y,Z,K)$ is a
solution of the reflected BSVIE associated with $(f,\xi,L)$ if it satisfies \eqref{rbsde-1} along with the properties (a), (b) and (c).
\end{definition}

\medskip
\begin{remark}
In the case of a standard reflected BSDEs, the above Skorohod flatness condition (c2) is equivalent to 
$$
\int_{0}^{T}(Y(s)-L(s))\,dK(s)=0,\quad \Pf\as
$$
where the integral is to be interpreted in the Lebesgue-Stieltjes sense.  It would be interesting to have a similar characterization for the Volterra type reflected equations.
\end{remark}

We make the following assumptions on $(f,\xi, L)$.

\subsection{Assumptions on $(f,\xi, L)$}

\begin{itemize}
\item[(A1)] The terminal condition $\xi(t)$, parametrized by $t$ is a $\B([0,T]\otimes \F_T$-measurable map  $\xi:\,\Omega\times[0,T]\longrightarrow \R$ which satisfies  
$$
\sup_{0\le t\le T}\E[|\xi(t)|^2]<\infty;
$$
\item[(A2)] The driver $f$ is a map from $\Omega\times[0,T]\times[0,T]\times \R\times \R$ onto $\R$ which satisfies, for any fixed $(t,y,z)\in[0,T]\times \R\times\R$, the process $f(t,\cdot,y,z)$ is progressively measurable. Moreover, 

\medskip
\begin{itemize}
\item[(A2-a)] $\quad
\underset{0\le t\le T}{\sup}\E\left[\left(\int_t^T |f(t,s,0,0)|ds\right)^2\right]<\infty;
$
\item[(A2-b)] There exists a positive constant $c_f$ such that $\Pf\as$, for all $(t,s)\in[0,T]^2$ and $y,y^{\prime},z,z^{\prime}\in\R$,
$$
|f(t,s,y,z)-f(t,s,y^{\prime},z^{\prime})|\le c_f\left(|y-y^{\prime}
|+|z-z^{\prime}|\right).
$$
\item[(A2-c)] For some $\alpha\in (0,1/2]$ and $c_1>0$, for all $(y,z)\in\R\times\R$ and all $0\le t , t^{\prime}\le s\le T$, 
$$
|f(t^{\prime},s,y,z)-f(t,s,y,z)|\le c_1|t^{\prime}-t|^{\alpha},
$$
and for some $\beta> 1/{\alpha}$ and $c_2>0$, 
$$
\E[|\xi(t)-\xi(t^{\prime})|^{\beta}]\le c_2|t^{\prime}-t|^{\alpha\beta}
$$
and
$$
\E\left[\left(\int_0^T|f(0,s,0,0)|^2ds\right)^{\beta/2}\right]<\infty.
$$
\end{itemize}
\item[(A3)] The obstacle $(L(u),\,0\le u\le T)$  is a real-valued and  $\Ff$-adapted continuous process satisfying
\begin{equation}\label{L}
 L(T)\le \xi(t),\,\,\, t\in [0,T]\quad\text{and}\quad  \E[\sup_{0\le u\le T}(L^+(u))^2]<\infty.
 \end{equation}
 \end{itemize}

\begin{remark} $ $
\begin{enumerate}
\item Assumption (A2-c) is similar to $(\widetilde{H1})$ in \cite{WZ} and implies that 
\begin{equation}\label{f-sup}
\E\left[\left(\underset{0\le t\le T}{\sup}\int_t^T |f(t,s,0,0)|^2ds\right)^{\beta/2}\right]<\infty
\end{equation}
which is stronger than (A2-a) since $\beta >\frac{1}{\alpha}>2$. 
\item As we will see it below, the assumptions (A1), (A2-a) and (A3) yield existence of a unique process $(Y,Z,K)$ which satisfies \eqref{rbsde-1} along with $(Y,Z)\in\hH^2\times\Lf^2$, Property (b) and $K$ satisfying (c1) except the continuity of $(t,u)\mapsto K(t,u)$.  Assumption (A2-c) yields the continuity of $Y$ and the bicontinuity of $K(\cdot,\cdot)$ which in turn insure the Skorohod flatness condition (c3). 
\end{enumerate}
\end{remark}
\section{Existence and Uniqueness of solutions to reflected BSVIEs}\label{EU}
In this section we derive existence, uniqueness and a comparison result for continuous reflected  BSVIEs.  We have

\begin{theorem}\label{main-1}
Suppose the assumptions (A1), (A2) and (A3) are satisfied. Then the reflected BSVIE
\eqref{rbsde-1} associated with $(f,\xi,L)$ admits a unique solution $(Y,Z,K)$ which satisfies (a), (b) and (c). Moreover, we have $\Pf\as$ the representation, for every $t\in[0,T]$, 
\begin{equation}\label{Y-rep}
Y(t)=\underset{\tau\ge t}{\text{ess sup}}\,\E\left[\int_t^{\tau}f(t,s, Y(s),Z(t,s))ds+L(\tau)\mathbb{1}_{\{\tau<T\}}+\xi(t)\mathbb{1}_{\{\tau=T\}}\Big|\,\F_t\right],
\end{equation}
where the essential supremum is taken over $\Ff$-stopping times $\tau$ taking values in $[0,T]$.
\end{theorem}
\begin{remark}
We note that the representation \eqref{Y-rep} does not imply that $Y$ is a supermartingale, as it can easily be checked.
\end{remark}

\bigskip
Inspired by the approach suggested in \cite{lin}, \cite{WZ} and \cite{Yong2} to solve the ordinary BSVIE by identifying an accompanying true martingale, we derive a unique solution to the reflected BSVIE \eqref{rbsde-1} associated with $(f,\xi,L)$ by applying the notion of Snell envelope along with a contraction argument to an accompanying supermartingale $\widetilde{Y}(t,\cdot)$, parametrized by $t$, defined below from which we obtain $Y$ by setting $Y(t):=\widetilde{Y}(t,t)$. To this end, we first consider the case where the driver $f$ does not depend on $(Y,Z)$. Then, we consider the general case where $f$ depends on $(Y,Z)$.

\subsection{Driver independent of Y and Z}
Consider the following reflected BSVIE
\begin{equation}\label{rbsde-const}
Y(t)=\xi(t)+\int_{t}^{T}f(t,s)ds+\int_t^TK(t,ds)-\int_{t}^{T} Z(t,s)dW(s),
\end{equation}
where the driver $f$ does not depend on $(Y,Z)$. We have

\begin{proposition}\label{const} Under Assumptions (A1) and (A2), there exists a unique solution $(Y,Z,K)$ to the reflected BSVIE \eqref{rbsde-const} associated with $(f,\xi,L)$ satisfying (a), (b) and (c) and admits the representation \eqref{Y-rep}.
\end{proposition}

\begin{proof}
The proof is based on existence and uniqueness of the process $(\widetilde{Y},Z,K)$ which satisfies the following reflected BSDE, parametrized by $t$, associated with $(f,\xi,L)$: for each fixed $t\in[0,T]$,
\begin{equation}\label{bsde-aux}
\widetilde{Y}(t,u)=\xi(t)+\int_{u}^{T}f(t,s)ds+\int_u^TK(t,ds)-\int_{u}^{T} Z(t,s)dW(s),\,\, u\in[0,T],
\end{equation}
where
\begin{itemize}
\item[(d)]  $\widetilde{Y}(t,\cdot)\in\cS^2,\,\,K(t,T)\in L^2(\Pf),\,\, Z(t,\cdot)\in\hH^2$;
\item[(e)] $\widetilde{Y}(t,u)\ge L(u)\quad \Pf\as,\quad 0\le u\le T$;
\item[(f)]  $K(t,\cdot)$ is continuous  and increasing, $K(t,0)=0$ and satisfies the Skorohod's flatness condition:
for each $0\le \alpha<\beta \le T$,
$$\begin{array}{lll}
K(t,\alpha)=K(t,\beta) \quad\text{whenever}\quad\widetilde{Y}(t,u)>L(u) \,\, \text{for each}\,\,\, u\in[\alpha, \beta] \quad \Pf\as,  
\end{array}
$$
which is equivalent to the property
$$
\int_{t}^{T}(\widetilde{Y}(t,u)-L(u))\,K(t,du)=0, \quad \Pf\as 
$$
\end{itemize}
A solution $(Y,Z,K)$ to \eqref{rbsde-const} is obtained from  $(\widetilde{Y},Z,K)$ by setting $Y(t):=\widetilde{Y}(t,t)$.

\medskip

$\bullet$ {\bf Existence of a solution}.  For a fixed $t\in[0,T]$, set
$$
\Gamma(t,u):=\int_t^uf(t,s)ds+L(u)\mathbb{1}_{\{u<T\}}+\xi(t)\mathbb{1}_{\{u=T\}},\quad u\in[0,T].
$$
In view of (A2-a) and the continuity of the obstacle process $L$, the map $u\mapsto \Gamma(t,u)$ is continuous. Moreover, by (A1) and (A3), it holds that, for each $t\in[0,T]$,
$$\begin{array}{lll}
\underset{0\le t\le T}{\sup}\E\left[\underset{0\le u\le T}{\sup}|\Gamma(t,u)|^2\right]\le 8\underset{0\le t\le T}{\sup}\E\left[|\xi(t)|^2+\left(\int_t^T|f(t,u)|du\right)^2 \right. \\ \left. \qquad\qquad\qquad \qquad\qquad\qquad+\underset{0\le s\le T}{\sup}(L^+(s))^2\right]<\infty.
\end{array}
$$
Consider the process $(\widetilde{Y}(t,u),\,u\in[0,T])$ defined by
\begin{equation}\label{Y-tilde-snell}
\widetilde{Y}(t,u):=\underset{\tau\ge u}{\text{ess sup}}\,\E\left[\int_u^{\tau}f(t,s)ds+L(\tau)\mathbb{1}_{\{\tau<T\}}+\xi(t)\mathbb{1}_{\{\tau=T\}}\Big|\,\F_u\right],
\end{equation}
where the essential supremum is taken over $\Ff$-stopping times $\tau$ taking values in $[0,T]$. The process $X_t(u):=\widetilde{Y}(t,u)+\int_t^u f(t,s)ds,\,\,0\le u\le T$ satisfies
$$
X_t(u)=\underset{\tau\ge u}{\text{ess sup}}\,\E\left[\int_t^{\tau}f(t,s)ds+L(\tau)\mathbb{1}_{\{\tau<T\}}+\xi(t)\mathbb{1}_{\{\tau=T\}}\Big|\,\F_u\right],\quad u\in[0,T],
$$
i.e. it is the Snell envelope of the processes $(\Gamma(t,u))_{0\le u\le T}$ which is the smallest continuous supermartingale, parametrized by $t$, which  dominates the continuous process $\Gamma(t,\cdot)$. 
Furthermore, by Doob's inequality it holds that 
\begin{equation}\label{Y-est}
\underset{0\le t\le T}{\sup}\,\E\left[\underset{0\le u\le T}{\sup}|\widetilde{Y}(t,u)|^2\right]\le C\underset{0\le t\le T}{\sup}\E\left[\underset{0\le u\le T}{\sup}|\Gamma(t,u)|^2\right]<\infty. 
\end{equation}
In particular, the processes $\Gamma(t,\cdot)$ and $\widetilde{Y}(t,\cdot)$ are uniformly integrable and the processes $\Gamma(\cdot,u)$ and $\widetilde{Y}(\cdot,u))$ are square-integrable.  Furthermore, the supermartingale $(\widetilde{Y}(t,u)+\int_t^u f(t,s)ds)_{0\le u\le T}$ is square-integrable. 

\medskip\noindent
By the well known techniques related the Snell envelope which use the Doob-Meyer decomposition along with the martingale representation theorem see e.g. \cite{elkaroui} or \cite{zhang}, there exists a unique adapted increasing continuous process $K(t,\cdot)$ such that $K(t,0)=0$ and $K(t,T)\in L^2(\Pf)$ and an $\Ff$-adapted process $Z(t,\cdot)\in\hH^2$ such that
\begin{equation*}
\widetilde{Y}(t,u)=\xi(t)+\int_{u}^{T}f(t,s)ds+\int_u^TK(t,ds)-\int_{u}^{T} Z(t,s)dW(s),\,\, u\in[0,T],
\end{equation*}
for which the properties (d), (e) and (f) are satisfied.  Moreover, the process  \\ $\int_0^u Z(t,s)dW(s),\,\, u\in[0,T],$ is a uniformly integrable martingale and 
\begin{equation}\label{Y-tilde-est-1}\begin{array}{lll}
\E\left[\underset{0\le u\le T}{\sup}|\widetilde{Y}(t,u)|^2+\int_t^T|Z(t,s)|^2ds+K^2(t,T)\right]\le C\E\left[|\xi(t)|^2 \right.\\ \qquad\qquad\qquad \qquad\qquad\qquad\qquad  \left. +\int_t^T|f(t,s)|^2ds +\underset{0\le u\le T}{\sup}\,(L^+(u))^2\right].
\end{array}
\end{equation}
Hence, since $Y(t)=\widetilde{Y}(t,t)$, it satisfies $\Pf\as$ the representation \eqref{Y-rep} and the following estimate for $(Y,Z,K)$ holds:
\begin{equation}\label{Y-Z-K-est}
\begin{array}{lll}
\E\left[\int_0^T|Y(t)|^2dt+\int_0^T\int_t^T|Z(t,s)|^2dsdt+\int_0^TK^2(t,T)dt\right] \\ 
\qquad\qquad \le C\E\left[\int_0^T|\xi(t)|^2dt +\int_0^T\int_t^T|f(t,s)|^2ds +T\underset{0\le u\le T}{\sup}\,(L^+(u))^2\right],
\end{array}
\end{equation}
which is finite by the assumptions (A1), (A2-a) and (A3). Thus, $(Y, Z, K(\cdot,T))\in \hH^2\times\Lf^2\times\hH^2$.

\medskip Next, we will show that the maps $(t,u)\mapsto \widetilde{Y}(t,u),\, \int_t^u f(t,s)ds, \,\,\int_{0}^{u} Z(t,s)dW(s)$ and $K(t,u)$ are continuous. 

\noindent First we note that, given $t, t^{\prime}\in [0,T]$,  denoting by $
\Delta\xi:=\xi(t)-\xi(t^{\prime}),\,\, \Delta\widetilde{Y}(u):=\widetilde{Y}(t,u)-\widetilde{Y}(t^{\prime},u)$,  $\Delta Z(u):=Z(t,u)-Z(t^{\prime},u),\,\, \Delta f(u):=f(t,u)-f(t^{\prime},u)$ and $\Delta K(ds):=K(t,ds)-K(t^{\prime},ds)$ and applying It{\^o}'s formula to $|\widetilde{Y}(t,u)-\widetilde{Y}(t^{\prime},u)|^2$ and the Skorohod flatness condition (f), we have, for any $0\le u\le T$, 

\begin{equation}\label{est-yz-0}\begin{array}{lll}
|\Delta\widetilde{Y}(u))|^2+\int_u^T |\Delta Z(s)|^2ds=|\Delta \xi|^2+2\int_u^T(\Delta\widetilde{Y}(s)\Delta f(s)ds \\ \qquad\qquad\qquad \qquad +2\int_u^T\Delta\widetilde{Y}(s)d\Delta K(ds)-2\int_u^T \Delta\widetilde{Y}(s)\Delta Z(s)dW(s)\\ \qquad\qquad\qquad \le |\Delta \xi|^2+2\int_u^T\Delta \widetilde{Y}(s)\Delta f(s)ds -2\int_u^T \Delta\widetilde{Y}(s)\Delta Z(s)dW(s).
\end{array}
\end{equation}
Taking expectation, we obtain
\begin{equation}\label{est-yz}\begin{array}{lll}
\E\left[|\widetilde{Y}(t,u)-\widetilde{Y}(t^{\prime},u)|^2+\int_u^T |Z(t,s)-Z(t^{\prime},s)|^2ds\right]\le \E\left[
 |\xi(t^{\prime})-\xi(t)|^2\right] \\ \qquad\qquad\qquad\qquad+ 2\E\left[\int_u^T (\widetilde{Y}(t,u)-\widetilde{Y}(t^{\prime},u))(f(t,s)-f(t^{\prime},s))ds \right].
 \end{array}
 \end{equation} 
 Moreover, by Burkholder-Davis-Gundy's inequality, (A2-c) and Young's inequality it follows that, for any $\beta>1$, there exists a positive constant $C_{\beta}$ depending only on $\beta, c_1$ and $T$ such that 
 \begin{equation}\label{z-ineq}\begin{array}{lll}
 \E\left[\left(\int_0^T |Z(t,s)-Z(t^{\prime},s)|^2ds\right)^{\beta/2}\right]\le  C_{\beta}\E\left[|\xi(t)-\xi(t^{\prime})|^{\beta}\right] \\  \qquad\qquad\qquad\qquad\qquad+\E\left[\underset{0\le u\le T}{\sup}|\widetilde{Y}(t,u)-\widetilde{Y}(t^{\prime},u)|^{\beta} \right]\\ \qquad\qquad\qquad\qquad\qquad+|t-t^{\prime}|^{\frac{\alpha\beta}{2}}\left(\E\left[\underset{0\le u\le T}{\sup}|\widetilde{Y}(t,u)-\widetilde{Y}(t^{\prime},u)|^{\beta}\right]\right)^{1/2}.
 \end{array}
 \end{equation}

\medskip
$\bullet\,$ \underline{Continuity of the map $(t,u)\mapsto \widetilde{Y}(t,u)$}. 
Using (A2-c) we have, for $0\le t, t^{\prime}, u\le T$, \begin{equation*}\begin{array}{lll}
|\widetilde{Y}(t,u)-\widetilde{Y}(t^{\prime},u)|=\Big|\underset{\tau\ge u}{\text{ess sup}}\,\E\left[\int_u^{\tau}f(t,s)ds+L(\tau)\mathbb{1}_{\{\tau<T\}}+\xi(t)\mathbb{1}_{\{\tau=T\}}\Big|\,\F_u\right]-\\ \qquad\qquad \qquad \underset{\tau\ge u}{\text{ess sup}}\,\E\left[\int_u^{\tau}f(t^{\prime},s)ds+L(\tau)\mathbb{1}_{\{\tau<T\}}+\xi(t)\mathbb{1}_{\{\tau=T\}}\Big|\,\F_u\right]\Big|\\ \qquad\qquad\qquad \qquad \qquad  \le \underset{\tau\ge u}{\text{ess sup}}\,\E\left[\int_u^{\tau}|f(t,s)-f(t^{\prime},s)|ds+|\xi(t)-\xi(t^{\prime})|\Big|\,\F_u\right]\\ \qquad\qquad\qquad \qquad \qquad  \le \E\left[\int_u^T|f(t,s)-f(t^{\prime},s)|ds+|\xi(t)-\xi(t^{\prime})|\Big|\,\F_u\right]
\\ \qquad\qquad\qquad \qquad \qquad  \le T  c_1|t^{\prime}-t|^{\alpha}+
\E\left[|\xi(t)-\xi(t^{\prime})|\Big|\,\F_u\right].
\end{array}
\end{equation*}
By Doob's maximal inequality and (A2-c), since $\beta>1/{\alpha}\ge 2$,
$$
\begin{array}{lll}
\E\left[\underset{0\le u\le T}{\sup}\left(\E\left[|\xi(t)-\xi(t^{\prime})|\Big|\,\F_u\right]\right)^{\beta}\right]\le \left(\frac{\beta}{\beta-1}\right)^{\beta} \E[|\xi(t)-\xi(t^{\prime})|^{\beta}]\\ \qquad\qquad\qquad\qquad\qquad\qquad\qquad\qquad\le c_2\left(\frac{\beta}{\beta-1}\right)^{\beta}|t-t^{\prime}|^{\alpha\beta}.
\end{array}
$$
Hence, 
\begin{equation}\label{Y-tilde-est}
\E\left[\underset{0\le u\le T}{\sup}|\widetilde{Y}(t,u)-\widetilde{Y}(t^{\prime},u)|^{\beta}\right]\le C|t^{\prime}-t|^{\alpha\beta},
\end{equation}
where $C:=2^{\beta}\left(T  c_1+c_2\left(\frac{\beta}{\beta-1}\right)^{\beta}\right)$.

\medskip\noindent
In view of Kolmogorov's continuity theorem (see \cite{RY}, Theorem I.2.1.), $(t,u)\mapsto \widetilde{Y}(t,u)$ is continuous (admits a bicontinuous modification). In particular, $t\mapsto Y(t):= \widetilde{Y}(t,t)$ is continuous.

\medskip
$\bullet\,$ \underline{Continuity of the map $(t,u)\mapsto g(t,u):=\int_t^uf(t,s) ds$}. For any $0\le t<t^{\prime}\le T$, we have
\begin{equation*}\begin{array}{lll}
\E\left[ \underset{0\le u\le T}{\sup}|g(t^{\prime},u)-g(t,u)|^{\beta}\right]\le C\E\left[\big| \int_t^{t^{\prime}}f(t,s)ds\big|^{\beta}+\left(\int_0^T|f(t^{\prime},s)-f(t,s)|ds\right)^{\beta}\right].
\end{array}
\end{equation*}
By H{\"o}lder's inequality and (A2-c), we have
\begin{equation*}
\begin{array}{lll}
\E\left[\big| \int_t^{t^{\prime}}f(t,s)ds\big|^{\beta}\right]\le |t^{\prime}-t|^{\beta/2}\E\left[\left( \int_0^T|f(t,s)|^2ds\right)^{\beta/2}\right]\\ \qquad\qquad\qquad\qquad\quad \le |t^{\prime}-t|^{\beta/2} \underset{0\le t\le T}{\sup}\E\left[\left(\int_0^T|f(t,s)|^2ds\right)^{\beta/2}\right]\le C|t^{\prime}-t|^{\beta/2},
\end{array}
\end{equation*}
where the constant $C$ is due \eqref{f-sup}.
Moreover, 
$$
\E\left[ \left(\int_0^T|f(t^{\prime},s)-f(t,s)|ds\right)^{\beta}\right]\le (c_1T)^{\beta}|t^{\prime}-t|^{\alpha\beta}.
$$
Therefore, since $\alpha\in (0,1/2]$, we have 
\begin{equation}\label{g}
\E\left[ \underset{0\le u\le T}{\sup}|g(t^{\prime},u)-g(t,u)|^{\beta}\right]\le C(1+|t^{\prime}-t|^{\beta(\frac{1}{2}-\alpha)})|t^{\prime}-t|^{\alpha\beta} \le C|t^{\prime}-t|^{\alpha\beta}.
\end{equation}
Since $\alpha\beta>1$, by Kolmogorov's continuity theorem $t\mapsto g(t,u):=\int_t^uf(t,s) ds$ is continuous (admits a bicontinuous modification).

\medskip
$\bullet\,$ \underline{Continuity of the map $(t,u)\mapsto \int_0^uZ(t,s)dW(s)$}.
By Burkholder-Davis-Gundy's inequality and \eqref{z-ineq}, we have
$$
\begin{array}{lll}
\E\left[\underset{0\le u\le T}{\sup}|\int_0^u \left(Z(t,s)-Z(t^{\prime},s)\right)dW(s)|^{\beta}\right]\le (\frac{\beta}{\beta-1})^{\beta}
\E\left[\left(\int_0^T |Z(t,s)-Z(t^{\prime},s)|^2ds\right)^{\beta/2}\right] \\ \qquad\qquad  
\le  (\frac{\beta}{\beta-1})^{\beta}C_{\beta}\E\left[|\xi(t)-\xi(t^{\prime})|^{\beta}+ \underset{0\le u\le T}{\sup}|\widetilde{Y}(t,u)-\widetilde{Y}(t^{\prime},u)|^{\beta} \right] \\  \qquad +(\frac{\beta}{\beta-1})^{\beta}C_{\beta} |t-t^{\prime}|^{\frac{\alpha\beta}{2}}\left(\E\left[\underset{0\le u\le T}{\sup}|\widetilde{Y}(t,u)-\widetilde{Y}(t^{\prime},u)|^{\beta}\right]\right)^{1/2}.

\end{array}
$$ 
Therefore, by (A2-c) and \eqref{Y-tilde-est}, it holds that
$$
\E\left[\underset{0\le u\le T}{\sup}|\int_0^u \left(Z(t,s)-Z(t^{\prime},s)\right)dW(s)|^{\beta}\right]\le\widehat{C}_{\beta}|t-t^{\prime}|^{\alpha\beta}.
$$
Hence, in view of Kolmogorov's continuity theorem, $(t,u)\mapsto \int_0^u Z(t,s)dW(s)$ is continuous (admits a bicontinuous modification). 

\medskip
$\bullet\,$ \underline{Continuity of the map $(t,u)\mapsto K(t,u)$}. This follows from the fact that
$$\begin{array}{lll}
K(t,u)=\widetilde{Y}(t,u)-\widetilde{Y}(t,0)-\int_0^u f(t,s)ds+\int_0^uZ(t,s)dW(s),\,\, (t,u)\in[0,T]^2,
\end{array}
$$
and the bicontinuity of each of the terms on the r.h.s.

\medskip
With $Y(t):=\widetilde{Y}(t,t)$, we obtain a solution$(Y,Z,K)$ to \eqref{rbsde-const} which satisfies (a), (b) and the property (c1). It remains to check the Skorohod flatness condition (c2). Indeed, in view of the property (f), for each fixed $t\in [0,T]$, it holds that, for each $0\le \alpha<\beta \le T$,
$$\begin{array}{lll}
K(t,\alpha)=K(t,\beta) \quad\text{whenever}\,\, \widetilde{Y}(t,u)>L(u) \,\, \text{for each}\,\,\, u\in[\alpha, \beta] \quad \Pf\as,  
\end{array}
$$
Thanks to the bicontinuity of $\widetilde{Y}(\cdot,\cdot)$ and $K(\cdot,\cdot)$, by sending $t$ to $u$ it holds that
 for each $0\le \alpha<\beta \le T$,
$$
\begin{array}{lll}
K(u,\alpha)=K(u,\beta) \quad\text{whenever}\,\, Y(u)=\widetilde{Y}(u,u)>L(u) \,\, \text{for each}\,\,\, u\in[\alpha, \beta] \quad \Pf\as  
\end{array}
$$

\medskip
$\bullet$ {\bf Uniqueness of the solution}. 
Let $(Y^{\prime},Z^{\prime},K^{\prime})$ be another solution to \eqref{rbsde-1} associated with $(f, \xi, L)$ satisfying (a), (b) and (c). Define $\widehat{Y}:=Y-Y^{\prime}, \widehat{Z}:=Z-Z^{\prime}$ and $\widehat{K}:=K-K^{\prime}$, and correspondingly $\widehat{\widetilde{Y}}:=\widetilde{Y}-\widetilde{Y^{\prime}}$.   Applying It\^o's formula to $|\widehat{\widetilde{Y}}|^2$ and taking expectation we obtain, for each fixed $t\in [0,T]$, 
$$
\E\left[|\widehat{\widetilde{Y}}(t,u)|^2+\int_u^T|\widehat{Z}(t,s)|^2ds\right]= 2\E\left[\int_u^T\widehat{\widetilde{Y}}(t,s)\widehat{K}(t,ds)\right],\quad u\in[0,T].
$$
But, by the flatness condition (f), we have $\int_u^T\widehat{\widetilde{Y}}(t,s)\widehat{K}(t,ds)\le 0 \,\,\Pf\as$ which implies that
$\E\left[|\widehat{\widetilde{Y}}(t,u)|^2\right]=0$ and $E\left[\int_u^T|\widehat{Z}(t,s)|^2ds\right]=0$ for all $u\in[0,T]$. Therefore, 
$\widetilde{Y}(t,\cdot)=\widetilde{Y^{\prime}}(t,\cdot)$ and $Z(t,\cdot)=Z^{\prime}(t,\cdot)$ and then $K(t,\cdot)=K^{\prime}(t,\cdot)\,\,\Pf\as$  Hence $Y=Y^{\prime} \,\,\Pf\as$, by the continuity of  the map $(t,u)\mapsto \widetilde{Y}(t,u)$. 
\end{proof}

\subsection{The general case. Proof of Theorem \ref{main-1}} Consider the Banach space $\Ec:=\hH^2\times \Lf^2$  of the $\R\times\R$-valued processes $(Y,Z)$ with values in $\R\times\R$ such that $Y\in\hH^2, Z\in \Lf^2$ endowed $\|(Y,Z)\|^2_{\Ec}=\|Y\|^2_{\hH^2}+\|Z\|^2_{\Lf^2}$. Define the map $\Phi$ from $\Ec$ onto itself as follows: For any $(U,V)\in\Ec, \,\, (Y,Z)=\Phi(U,V)$ is the unique element of $\Ec$ such that if we define
\begin{equation}\label{K-fx}\begin{array}{lll}
\int_0^t K(t,ds):=Y(t)-Y(0)-\int_0^t f(t,s,U(s),V(t,s))ds\\ \qquad\qquad\qquad\qquad\qquad\qquad\qquad\qquad\qquad +\int_0^tZ(t,s)dW(s),\,\, t\in[0,T],
\end{array}
\end{equation}
then the triple $(Y,Z,K)$ solves the reflected BSVIE \eqref{rbsde-1} associated with the entrees $(f(t,s,U(s),V(t,s)),\xi,L)$.  

\medskip\noindent
Denote $\Ec([u,v])$ the Banach space $\Ec$ where the time horizon is restricted to $[u,v],\,\, u,v\in[0,T]$.

\begin{proposition}\label{delta} Assume (A1), (A2) and (A3).  Then there exists $\delta>0$ depending only on the Lipschitz constant of $f$ such that $\Phi$ is a contraction mapping on the space $\Ec([T-\delta,T])$.
\end{proposition}

\begin{proof} Let $X:=(U,V)$ and $X^{\prime}:=(U^{\prime},V^{\prime})$ be two elements of $\Ec$ and define
$\overline{X}=(\overline{U},\overline{V})=X-X^{\prime},\,\, \overline{Y}=Y-Y^{\prime}, \,\,\overline{Z}=Z-Z^{\prime}, \,\,\overline{K}=K-K^{\prime}$ and  for $s,t \in[0,T]$, $\bar{f}(t,s):=f(t,s, U(s),V(t,s))-
f(t,s, U^{\prime}(s),V^{\prime}(t,s))$. We have
$$
\int_0^t \overline{K}(t,ds):=\overline{Y}(t)-\overline{Y}(0)-\int_0^t \bar{f}(t,s)ds+\int_0^t\overline{Z}(t,s)dW(s),\quad 0\le t \le T.
$$
From Proposition \ref{const}, we have $\Pf\as$
$$\begin{array}{lll}
Y(t)=\underset{\tau\ge t}{\text{ess sup}}\,\E\left[\int_t^{\tau}f(t,s, U(s),V(t,s))ds+L(\tau)\mathbb{1}_{\{\tau<T\}}+\xi(t)\mathbb{1}_{\{\tau=T\}}\Big|\,\F_t\right], \\ 
Y^{\prime}(t)=\underset{\tau\ge t}{\text{ess sup}}\,\E\left[\int_t^{\tau}f(t,s, U^{\prime}(s),V^{\prime}(t,s))ds+L(\tau)\mathbb{1}_{\{\tau<T\}}+\xi(t)\mathbb{1}_{\{\tau=T\}}\Big|\,\F_t\right].
\end{array}
$$
Thus,
$$
\left|\overline{Y}(t)\right|\le \underset{\tau\ge t}{\text{ess sup}}\,\E\left[\int_t^{\tau}|\bar{f}(t,s)|ds\Big|\,\F_t\right]\le \E\left[\int_t^T|\bar{f}(t,s)|ds\Big|\,\F_t\right].
$$
Noting that by the Cauchy-Schwarz inequality and  (A2-b) we have
\begin{equation}\label{bar-f-est}\begin{array}{lll}
\E\left[\int_{T-\delta}^T\left(\int_t^T|\bar{f}(t,s)|ds\right)^2dt\right] \le  \delta 
\E\left[\int_{T-\delta}^T\int_t^T|\bar{f}(t,s)|^2ds dt\right] \\  \qquad\qquad \le 4c^2_f\delta 
\E\left[\int_{T-\delta}^T\int_t^T|\overline{U}(s)|^2ds dt+\int_{T-\delta}^T\int_t^T|\overline{V}(t,s)|^2ds dt\right] \\ \qquad\qquad 
\le 4c^2_f(\delta^2+\delta)\|(\overline{U},\overline{V})\|^2_{\Ec([T-\delta,T]}.
\end{array}
\end{equation}
Therefore, 
$$\begin{array}{lll}
\E\left[\int_{T-\delta}^T|\overline{Y}(t)|^2dt\right]\le \E\left[\int_{T-\delta}^T\left(\E\left[\int_t^T|\bar{f}(t,s)|ds\Big|\,\F_t\right]\right)^2dt\right] \\ \qquad\qquad \le  \E\left[\int_{T-\delta}^T\left(\int_t^T|\bar{f}(t,s)|ds\right)^2dt\right]   \le 4c^2_f(\delta^2+\delta)\|(\overline{U},\overline{V})\|^2_{\Ec([T-\delta,T]}.
\end{array}
$$
On the other hand, by It\^o's formula applied to $|\overline{\widetilde{Y}}(t,u)|^2$, taking expectation and using Young's inequality, we obtain
$$
\E\left[\int_u^T|\overline{Z}(t,s)|^2ds\right]\le 2 E\left[\int_u^T \overline{\widetilde{Y}}(t,s)\bar{f}(t,s)ds\right].
$$
In view of \eqref{Y-tilde-snell}, we have for any $u\in[0,T]$,
$$
|\overline{\widetilde{Y}}(t,u)|\le \underset{\tau\ge u}{\text{ess sup}}\,\E\left[\int_u^{\tau}|\bar{f}(t,s)|ds\Big|\,\F_u\right]\le \E\left[\int_u^T|\bar{f}(t,s)|ds\Big|\,\F_u\right].
$$
Therefore,
$$
\begin{array}{lll}
\E\left[\int_u^T|\overline{Z}(t,s)|^2ds\right]\le 2 E\left[\int_u^T \overline{\widetilde{Y}}(t,s)\bar{f}(t,s)ds\right] \\ \qquad\qquad \qquad\qquad
\le 2 E\left[\int_u^T \bar{f}(t,s)\E\left[\int_s^T|\bar{f}(t,r)|dr\Big|\,\F_s\right]ds\right]\\ \qquad\qquad \qquad\qquad
\le 2 E\left[\int_u^T |\bar{f}(t,s)|\E\left[\int_u^T|\bar{f}(t,r)|dr\Big|\,\F_s\right]ds\right]\\ \qquad\qquad \qquad\qquad 
= 2 E\left[\left(\int_u^T|\bar{f}(t,r)|dr\right)^2\right].
\end{array}
$$
In particular, in view of \eqref{bar-f-est}, 
$$\begin{array}{lll}
\E\left[\int_{T-\delta}^T\int_t^T|\overline{Z}(t,s)|^2dsdt\right]\le 2 E\left[\int_{T-\delta}^T\left(\int_t^T|\bar{f}(t,r)|dr\right)^2dt\right]\\ \qquad\qquad \qquad\qquad \le 8c^2_f(\delta^2+\delta)\|(\overline{U},\overline{V})\|^2_{\Ec([T-\delta,T])}.
\end{array}
$$
Summing up, we have
$$
\|(\overline{Y},\overline{Z})\|^2_{\Ec([T-\delta,T]}\le 8c_f(\delta^2+\delta)\|(\overline{U},\overline{V})\|^2_{\Ec([T-\delta,T]}.
$$
Now, choosing $\delta>0$ such that  $\delta^2+\delta<1/8c_f$,  yields that  $\Phi$ is a contraction mapping on $\Ec([T-\delta,T]$ and thus admits a unique fixed point which yields  the unique solution of the reflected BSVIE \eqref{rbsde-1} associated with $(f,\xi,L)$ over $[T-\delta,T]$.
\end{proof}

\noindent We are now ready to give a proof of Theorem \ref{main-1}.
\begin{proof} By repeatedly applying the fixed point argument of Proposition \ref{delta} over adjacent time intervals of fixed length  $\delta$, satisfying $\delta^2+\delta<1/8c_f$, and pasting the solutions that we describe below,  we finally obtain existence of a unique solution in $\hH^2\times \Lf^2\times \cK^2$  to the reflected BSVIE \eqref{rbsde-1} over the whole time interval $[0,T]$.  

Since we are dealing with Volterra-type integrals where the driver $ f$ depends on $t$, we paste two adjacent solutions as follows. Fix such a $\delta$ and let $(Y^0, Z^0, K^0)$  be the unique solution of \eqref{rbsde-1} on $\Ec([T-\d, T])$. Consider the accompanying reflected BSDE $(\widetilde{Y}^0(t,\cdot), Z^0(t,\cdot),K^0(t,\cdot))$, with lower obstacle $L$, parametrized by $t\in [0,T]$, defined by 
\begin{equation*}\begin{array}{lll}
\widetilde{Y}^0(t,u)=\xi(t)+\int_{u}^{T}f(t,s, Y^0(s),\widetilde{Z}^0(t,s))ds \\ \qquad\qquad\qquad\qquad\qquad \qquad\qquad +\int_u^T\widetilde{K}^0(t,ds)-\int_{u}^{T} \widetilde{Z}^0(t,s)dW(s),\,\,\, u\in[T-\delta,T],
\end{array}
\end{equation*}
where $\widetilde{K}^0(t,T-\delta)=0$. In particular, $\widetilde{Y}^0(t,T-\delta)\in L^2(\F_{T-\delta})$. Moreover,  by Proposition \ref{delta},  we have $\Pf$-a.s.
$$
(\widetilde{Y}^0(t,t), \widetilde{Z}^0(t,s), \widetilde{K}^0(t,s))=(Y^0(t), Z^0(t,s), K^0(t,s)),\quad (t,s)\in [T-\delta,T]^2,\, t\le s.
$$

\noindent On the time interval $[T-2\delta,T-\delta]$, we consider the BSVIE 
\begin{equation}\label{2d}\begin{array}{ll}
Y^1(t)=\widetilde{Y}^0(t,T-\delta)+\int_{t}^{T-\delta}f(t,s,Y^1(s),Z^1(t,s))ds\\ \qquad\qquad\qquad +\int_t^{T-\delta}K^1(t,ds)-\int_{t}^{T-\delta} Z^1(t,s)dW(s),\,\, t\in[T-2\delta,T-\delta].
\end{array}
\end{equation}
Since the length of the time interval $[T-2\delta,T-\delta]$ is $\delta$, we may apply Proposition \ref{delta} to obtain a unique solution $(Y^1,Z^1,K^1)\in \hH^2([T-2\delta,T-\delta])\times \Lf^2([T-2\delta,T-\delta])\times\cK^2([T-2\delta,T-\delta])$ to the reflected BSVIE \eqref{2d}. 

\noindent Now, with
$$
Y(t):=\begin{cases} Y^0(t),  & t\in[T-\delta,T], \\ Y^1(t), & t\in[T-2\delta,T-\delta],
\end{cases}
$$
and
$$ 
(Z(t,s),K(t,s)):=\begin{cases} (Z^0(t,s),K^0(t,s)),  & (t,s)\in[T-\delta,T]^2,\,\, t\le s, \\ 
(\widetilde{Z}^0(t,s),\widetilde{K}^0(t,s)),  & t\in[T-2\delta,T-\delta]\times [T-\delta,T],\\
(Z^1(t,s),K^1(t,s)), & (t,s)\in[T-2\delta,T-\delta]^2, \,\, t\le s,
\end{cases}
$$
 the process $(Y,Z,K)$ solves the reflected BSVIE \eqref{rbsde-1} over the time interval $[T-2\delta,T]$.

\noindent By repeating the same reasoning on each time interval $[T-(m+1)\delta,T-m\delta],\,\, m=1,2,\ldots,n$ (where $n$ is arbitrary) with a similar dynamics but terminal condition $\widetilde{Y}^{m-1}(t,T-m\delta)$ at time $T-m\delta$, for $t\in [T-(m+1)\delta,T-m\delta]$, we construct recursively, for $m=1,2,\ldots,n$,  a solution $(Y^m,Z^m,K^m)\in \hH^2\times \Lf^2\times \cK^2$ on each time interval $[T-(m+1)\delta,T-m\delta]$. Pasting these processes, we obtain a unique solution $(Y,Z,K)$ of the reflected BSVIE \eqref{rbsde-1} on the full time interval $[0,T]$.\\ 
\end{proof}

A closer look at the way the estimate \eqref{Y-tilde-est-1} is derived in e.g. \cite{elkaroui} or \cite{zhang}, we obtain the following proposition as a direct consequence of the estimate leading to it.

\begin{proposition}\label{Y-S}
If instead of (A1) and (A2), we assume that 
$$
\E[\sup_{0\le t\le T}|\xi(t)|^2]<\infty \quad\text{and}\quad \E\left[\sup_{0\le t\le T}\left(\int_t^T |f(t,s,0,0)|ds\right)^2\right]<\infty,
$$
then 
$$
\|Y\|^2_{\cS^2}:=\E\left[\sup_{0\le t\le T}|Y(t)|^2\right]<\infty.
$$
\end{proposition}

\subsection{A comparison result for reflected BSVIEs}
In this section we derive a comparison theorem, similar to that of \cite{elkaroui}, Theorem 4.1 for standard reflected BSDEs, which extends  \cite{wy}, Theorem  3.4., for non-reflected BSVIEs. We follow the method of the proof in \cite{wy} and first derive a comparison when the driver $f$ does not depend of $y$, extending \cite{wy}, Proposition 3.3,  since in this case the proof is based on the comparison principle for reflected BSDEs, and then proceed to the proof of the general case using an approximation scheme. Some the imposed conditions on the coefficients can be relaxed at the expense of involved technical details  that we omit to make the content easy to follow. 

\begin{proposition}\label{comp}
Let $(\xi,f,L)$ and $(\xi^{\prime}, f^{\prime}, L^{\prime})$ two sets of entrees which satisfy the assumptions (A1), (A2) and (A3) and suppose further that 
\begin{itemize}
\item[(i)]  $\xi(t)\le \xi^{\prime}(t), \,\, \Pf\as, \,\,\, 0\le t\le T$,
\item[(ii)] $f(t,s,z)\le f^{\prime}(t,s,z), \quad   \text{for all}\,\,\, (t,z)\in[0,s]\times \R,\,\,\, \text{a.s., a.e.}\,\, s\in [0,T]$,
\item[(iii)] $L(s)\le L^{\prime}(s),\,\, 0\le s\le T,\,\, \Pf\as$ 
\end{itemize}
Let $(Y,Z,K)$ be a solution of the reflected BSVIE associated with $(\xi,f,L)$ and $(Y^{\prime},Z^{\prime},K^{\prime})$ be a solution of the reflected BSVIE associated with $(\xi^{\prime}, f^{\prime}, L^{\prime})$. Then
$$
Y(t)\le Y^{\prime}(t), \quad 0\le t\le T,\quad \Pf\as
$$
\end{proposition}

\begin{proof} For each fixed $t\in[0,T]$, let $\widetilde{Y}(t,\cdot)$ and $\widetilde{Y^{\prime}}(t,\cdot)$ be the standard BSDEs accompanying $Y$ and $Y^{\prime}$ respectively, constructed in a similar way to the one described in the proof of Proposition \ref{const}: For each fixed $t\in [0,T]$,
\begin{equation*}\begin{array}{lll}
\widetilde{Y}(t,u)=\xi(t)+\int_{u}^{T}f(t,s,Z(t,s))ds \\ \qquad\qquad\qquad\qquad\qquad \qquad +\int_u^TK(t,ds)-\int_{u}^{T} Z(t,s)dW(s),\,\,\, u\in[0,T].
\end{array}
\end{equation*}
A similar form for $\widetilde{Y^{\prime}}$ holds. We may apply the comparison theorem (\cite{elkaroui}, Theorem 4.1) for standard reflected BSDEs to obtain that, for each fixed $t\in[0,T]$,
$$
\widetilde{Y}(t,u)\le \widetilde{Y^{\prime}}(t,u), \quad 0\le u\le T\quad \text{a.s.}
$$ 
Since, the maps $(t,u)\mapsto \widetilde{Y}(t,u),\,\widetilde{Y^{\prime}}(t,u)$ are continuous, it follows that

$$
Y(t):=\widetilde{Y}(t,t)\le \widetilde{Y^{\prime}}(t,t)=:Y^{\prime}(t), \quad 0\le t\le T\quad \text{a.s.}
$$
\end{proof}

\begin{theorem}\label{compa-general}
Let $(\xi,f,L)$ and $(\xi^{\prime}, f^{\prime}, L^{\prime})$ two sets of entrees which satisfy the assumptions (A1), (A2) and (A3) and that either the map $y\mapsto f(t,s,y,z)$ or $y\mapsto f^{\prime}(t,s,y,z)$  is nondecreasing. Assume further that

\begin{itemize}
\item[(i)]  $\xi(t)\le \xi^{\prime}(t), \,\, \Pf\as, \,\,\, 0\le t\le T$,
\item[(ii)] $f(t,s,y,z)\le  f^{\prime}(t,s,y,z), \,\,\,   \text{for all}\,\,\, (t,z)\in[0,s]\times \R,\,\,\, \text{a.s., a.e.}\,\, s\in [0,T]$,
\item[(iii)] $L(s)\le L^{\prime}(s),\,\, 0\le s\le T,\,\, \Pf\as$ 
\end{itemize}
Let $(Y,Z,K)$ be a solution of the reflected BSVIE associated with $(\xi,f,L)$ and $(Y^{\prime},Z^{\prime},K^{\prime})$ be a solution of the reflected BSVIE associated with $(\xi^{\prime}, f^{\prime}, L^{\prime})$. Then
$$
Y(t)\le Y^{\prime}(t), \quad 0\le t\le T,\quad \Pf\as
$$
\end{theorem}

\begin{proof}
The proof follows the same steps of the proof of Theorem 3.2 in \cite{wy}. We sketch it and leave some technical details related to Peng's monotone convergence theorem \cite{peng} which are by now standard in the literature related to reflected BSDEs.  

Assume the map $y\mapsto f(t,s,y,z)$ is nondecreasing. Set $Y_0(\cdot)=Y^{\prime}(\cdot)$ and for $n\ge 1$, consider, the following sequence of reflected BSVIEs with the same lower obstacle $L$
\begin{equation*}
Y_n(t)=\xi(t)+\int_{t}^{T}f(t,s,Y_{n-1}(s),Z_n(t,s))ds+\int_t^TK_n(t,ds)-\int_{t}^{T} Z_n(t,s)dW(s). 
\end{equation*}
In view of the assumption (ii) and the monotonicity of $f$ in $y$, we may apply Proposition \eqref{comp}, to obtain that,  for every $n\ge 1$,
$$
Y^{\prime}(t)=Y_0(t)\ge Y_1(t) \ge \cdots \ge Y_n(t)\ge Y_{n+1}(t),\quad t\in [0,T],\,\, \text{a.s.}
$$
If we show that $(Y_n,Z_n,K_n)_{n\ge 1}$ converges to $(Y,Z,K)$ w.r.t. the norm defined, for a given $\theta >0$,  by
$$
\|(Y,Z,K)\|^2_{\theta}:=\E\left[\int_0^Te^{\theta t}|Y(t)|^2dt+\int_0^Te^{\theta t}\int_t^T|Z(t,s)|^2ds dt+\int_0^Te^{\theta t}|K(t,T)|^2dt\right]
$$
 which is equivalent with the norm defined on the space $\hH^2\times\Lf^2\times\hH^2$, then the limit process $(Y^*,Z^*,K^*)$
satisfies the reflected BSVIE 
$$
Y^*(t)=\xi(t)+\int_{t}^{T}f(t,s,Y^*(s),Z^*(t,s))ds+\int_t^TK^*(t,ds)-\int_{t}^{T} Z^*(t,s)dW(s)
$$
along with the Skorohod flatness condition (obtained by taking a.s. converging subsequences) w.r.t. the same lower obstacle $L$. By uniqueness of the solution,  we have $Y^*(\cdot)=Y(\cdot)$ which in turn entails the comparison principle. 
 
To show convergence,  let $\widetilde{Y}_n(t,\cdot)$  the solution to the standard reflected BSDEs, parametrized by $t \in [0,T]$,
  accompanying $Y_n$ (i.e.  $\widetilde{Y}_n(t,t)=Y_n(t)$), constructed in a similar way to the one described in the proof of Proposition \ref{const}, defined, for each fixed $t\in [0,T]$, by
\begin{equation*}\begin{array}{lll}
\widetilde{Y}_n(t,u)=\xi(t)+\int_{u}^{T}f(t,s,Y_{n-1}(s), Z_n(t,s))ds \\ \qquad\qquad\qquad\qquad\qquad \qquad +\int_u^TK_n(t,ds)-\int_{u}^{T} Z_n(t,s)dW(s),\,\,\, u\in[0,T],
\end{array}
\end{equation*}
For $n, m\ge1$, set 
$\delta\widetilde{Y}(t,s):=\widetilde{Y}_n(t,s)-\widetilde{Y}_m(t,s)$, $\delta Z(t,u):=Z_n(t,u)-Z_m(t,u)$, $\delta f(t,s):=f(t,s,Y_{n-1}(s), Z_n(t,s))-f(t,s,Y_{m-1}(s), Z_m(t,s))$ and $\delta K(t,u)=K_n(t,u)-K_m(t,u)$. 

\noindent For any $\theta>0$, we may apply It{\^o}'s formula to $e^{\theta u}|\delta\widetilde{Y}(t,u)|^2$ along with the Skorohod flatness condition to obtain
\begin{eqnarray*}\begin{array}{lll}
\E\left[e^{\theta u}|\delta\widetilde{Y}(t,u)|^2+\int_u^T e^{\theta s}|\delta Z(t,s))|^2ds\right] \\ \qquad\qquad
\le 2\E\left[\int_u^T e^{\theta s} \delta\widetilde{Y}(t,s)\delta f(t,s)ds\right]-\theta \E\left[\int_u^Te^{\theta s}|\delta\widetilde{Y}(t,s)|^2ds\right]\\ \qquad\qquad
\le -\theta\E\left[\int_u^Te^{\theta s}|\delta\widetilde{Y}(t,s)|^2ds\right]+\E\left[\int_u^T e^{\theta s}\left(\theta |\delta\widetilde{Y}(t,s)|^2+\frac{1}{\theta}|\delta f(t,s)|^2\right)ds\right] \\ \qquad\qquad
\le \frac{2 c^2_f}{\theta}\E\left[\int_u^T e^{\theta s}\left(|Y_{n-1}(s)-Y_{m-1}(s)|^2+|\delta Z(t,s)|^2\right)ds\right].
\end{array}
\end{eqnarray*}
Therefore, with $Y_n(t)=\widetilde{Y}_n(t,t)$, we have
\begin{eqnarray*}\begin{array}{lll}
E\left[\int_0^Te^{\theta t}|Y_{n}(t)-Y_{m}(t)|^2dt\right]+(1-\frac{2c^2_f}{\theta})\E\left[\int_0^T\int_t^T e^{\theta s}|Z_n(t,s)-Z_m(t,s))|^2dsdt\right] \\ \qquad\qquad\qquad\qquad
\le \frac{2Tc_f^2}{\theta}\E\left[\int_0^T e^{\theta s}|Y_{n-1}(s)-Y_{m-1}(s)|^2ds\right].
\end{array}
\end{eqnarray*}
Moreover, in view of the estimate in Proposition 3.6 of \cite{elkaroui}, there exists a constant $C$ independent of $n$ and $m$ such that 

\begin{eqnarray*}\begin{array}{lll}
\E\left[|\delta K(t,T)|^2\right]\le C \E\left[\int_t^T |\delta f(t,s)|^2ds\right]  \\ \qquad\qquad \qquad
\le 2C c_f^2 \E\left[\int_t^T\left(|Y_{n-1}(s)-Y_{m-1}(s)|^2+|\delta Z(t,s)|^2\right)ds\right].
\end{array}
\end{eqnarray*}
Therefore,
\begin{eqnarray*}\begin{array}{lll}
\E\left[\int_0^Te^{\theta t}|\delta K(t,T)|^2dt\right]\le 2C c_f^2\E\left[\int_0^Te^{\theta t}\int_t^T|Y_{n-1}(s)-Y_{m-1}(s)|^2dsdt\right]\\ \qquad\qquad\qquad\qquad\qquad\qquad\qquad\qquad  +2C c_f^2\E\left[\int_0^Te^{\theta t}\int_t^T|\delta Z(t,s)|^2dsdt\right]\\ 
\le  \frac{2C c_f^2}{\theta}\E\left[\int_0^Te^{\theta s}|Y_{n-1}(s)-Y_{m-1}(s)|^2ds\right]+2C c_f^2\E\left[\int_0^Te^{\theta t}\int_t^T|\delta Z(t,s)|^2dsdt\right],
\end{array}
\end{eqnarray*}
since
$$
\begin{array}{lll}
\E\left[\int_0^Te^{\theta t}\int_t^T|Y_{n-1}(s)-Y_{m-1}(s)|^2dsdt\right]= \E\left[\int_0^T|Y_{n-1}(s)-Y_{m-1}(s)|^2ds\int_0^se^{\theta t}dt\right] \\ \qquad\qquad\qquad\qquad\qquad\qquad\qquad\qquad\quad \le \frac{1}{\theta}\E\left[\int_0^Te^{\theta s}|Y_{n-1}(s)-Y_{m-1}(s)|^2ds\right].
\end{array}
$$
Thus, by choosing $\theta> 2c^2_f(1+2T)$, it follows that $(Y_n,Z_n,K_n)_n$ is a Cauchy sequence w.r.t. the norm $\|\cdot\|_{\theta}$. 
\end{proof}

\section{Application to time-inconsistent optimal stopping problems}\label{OS}
Let $(X(t), t\in [0,T])$ be an $\Ff$-adapted process. In many financial applications $X$ may model the price of a commodity.  We propose to solve the following time-inconsistent optimal stopping problem associated with $(f,L,\xi)$ which satisfies the assumptions (A1), (A2) and (A3) where e.g. $\xi(t):=h(t,X(T)$:
\begin{equation}\label{os-1}
\underset{\tau\ge t}{\text{sup}}\, J(t,\tau),
\end{equation}
for
\begin{equation}\label{os-J}
J(t,\tau):=\E\left[\int_t^{\tau}f(t,s, X(s))ds+L(\tau)\mathbb{1}_{\{\tau<T\}}+\xi(t)\mathbb{1}_{\{\tau=T\}}\right], 
\end{equation}
where the supremum is taken over $\Ff$-stopping times $\tau$ taking values in $[t,T]$.  More precisely, we would like to find an  $\Ff$ stopping time  $\tau_t^*$, indexed by $t$, such that 
$$
\tau^*_t=\underset{\tau\ge t}{\arg\max}\, J(t,\tau).
$$

The dependence of both the driver $f$ and the terminal value $\xi$ on $t$ makes the optimal stopping problem \eqref{os-1} time-inconsistent because then the associated value-function does not satisfy the dynamic programming principle i.e. it is not a supermartingale. 

Define  $(Y(t), t\in [0,T])$ by 
\begin{equation}\label{os-2}
Y(t):=\underset{\tau\ge t}{\text{ess sup}}\,\E\left[\int_t^{\tau}f(t,s, X(s))ds+L(\tau)\mathbb{1}_{\{\tau<T\}}+\xi(t)\mathbb{1}_{\{\tau=T\}}\Big|\,\F_t\right]. 
\end{equation}

In time-consistent optimal stopping problems (i.e. when $f$ and $\xi$ do not depend of $t$), the process $Y$ would be the value function of the stopping problem \eqref{os-1}. In particular $Y(0)=\underset{\tau\ge 0}{\text{sup}}\, J(0,\tau)$. This is not the case here. However, we have
$$
\underset{\tau\ge t}{\text{sup}}\, J(t,\tau)\le \E[Y(t)]. 
$$
Now, if we can find an $\Ff$-stopping time $\tau_t^*$ such that
$$
Y(t)=\E\left[\int_t^{\tau^*_t}f(t,s, X(s))ds+L(\tau^*_t)\mathbb{1}_{\{\tau^*_t<T\}}+\xi(t)\mathbb{1}_{\{\tau^*_t=T\}}\Big|\,\F_t\right]
$$
it follows that $\tau^*_t$ an optimal strategy for $J(t,\cdot)$ since
\begin{equation}\label{os-3}
J(t,\tau^*_t)=\E[Y(t)]\le \underset{\tau\ge t}{\text{sup}}\, J(t,\tau)\le \E[Y(t)].
\end{equation}

In view of Proposition \ref{const}, there exists a unique process $(Z,K)\in \Lf^2\times \hH^2$ such that 
$$
Y(t)=\xi(t)+\int_{t}^{T}f(t,s, X(s))ds+\int_t^TK(t,ds)-\int_{t}^{T} Z(t,s)dW(s),
$$
and $(Y,Z,K)$ satisfies the properties (a), (b) and (c) of Section \ref{prem}. Moreover, $Y$ is constructed so that, for every $t\in [0,T]$, $Y(t)=\widetilde{Y}(t,t)$, where $\widetilde{Y}(t,\cdot)$ is the accompanying reflected BSDE associated with $(f(t,s,X(s)), L(s),\xi(t))$, parametrized by $t$, which satisfies
$$
\widetilde{Y}(t,u)=\xi(t)+\int_{u}^{T}f(t,s, X(s))ds+\int_u^TK(t,ds)-\int_{u}^{T} Z(t,s)dW(s),\quad u\in [t,T],
$$
and for which $(\widetilde{Y}(t,\cdot),Z(t,\cdot), K(t,\cdot))$ satisfies the conditions (d), (e) and (f) displayed in the proof of Proposition \ref{const}, above. 

\begin{proposition}\label{opts} Suppose the assumptions (A1), (A2) and (A3) are satisfied. For each $t\in [0,T]$, denote by $\tau^*_t$ the stopping time
$$
\tau^*_t=\inf\{t\le u\le T;\,\, \widetilde{Y}(t,u)=L(u) \}
$$
with the convention that $\tau^*_t=T$ if $\, \widetilde{Y}(t,u)>L(u),\,\, t\le u\le T$.

Then $\tau^*_t$ is optimal in the sense that
\begin{equation}
Y(t)=\E\left[\int_t^{\tau^*_t}f(t,s, X(s))ds+L(\tau^*_t)\mathbb{1}_{\{\tau^*_t<T\}}+\xi(t)\mathbb{1}_{\{\tau^*_t=T\}}\Big|\,\F_t\right].
\end{equation}
Moreover, $\tau^*_t$  is an optimal  strategy for $J(t,\cdot)$ i.e. 
$$
\tau^*_t=\underset{\tau\ge t}{\arg\max}\, J(t,\tau).
$$
\end{proposition}

\begin{proof} We have
$$
\begin{array}{lll}
Y(t)=\xi(t)+\int_{t}^{T}f(t,s,X(s))ds+\int_t^TK(t,ds)-\int_{t}^{T} Z(t,s)dW(s) \\ \qquad 
=\widetilde{Y}(t,\tau^*_t)+\int_{t}^{\tau^*_t}f(t,s,X(s))ds+\int_t^{\tau^*_t}K(t,ds)-\int_{t}^{\tau^*_t} Z(t,s)dW(s).
\end{array}
$$
Now, Skorohod flatness condition (f) along with the continuity of the map $u\mapsto K(t,u)$ imply that 
$$
\int_t^{\tau^*_t} K(t,ds)=0.
$$
Thus, taking conditional expectation we finally obtain 
$$
Y(t)=\E\left[\int_t^{\tau^*_t}f(t,s, X(s))ds+L(\tau^*_t)\mathbb{1}_{\{\tau^*_t<T\}}+\xi(t)\mathbb{1}_{\{\tau^*_t=T\}}\Big|\,\F_t\right],
$$ 
since, by continuity of the map $u\mapsto\widetilde{Y}(t,u)$, we have  $\widetilde{Y}(t,\tau^*_t)=L(\tau^*_t)\mathbb{1}_{\{\tau^*_t<T\}}+\xi(t)\mathbb{1}_{\{\tau^*_t=T\}}$. In view of \eqref{os-3}, it follows that $\tau^*_t=\underset{\tau\ge t}{\arg\max}\, J(t,\tau).$

\end{proof}

\begin{remark}
The choice of stopping time $\tau_t^*$ as the first hitting time of the accompanying  Snel envelope $\widetilde{Y}(t,\cdot)$ of the obstacle $L$ instead of the value function $Y$, as it is the case for standard reflected BSDEs, is simply due to fact that for Volterra type equations we have 
$$
Y(t)\neq Y(u)+\int_u^Tf(t,s, Y(s),Z(t,s))ds+\int_u^TK(t,ds)-\int_u^T Z(t,s)dW(s),\quad u\ge t.
$$

\end{remark}


\begin{thebibliography}{9}                                                                                                

\bibitem {Ag} Agram, N. (2019). Dynamic risk measure for BSVIE with jumps and
semimartingale issues. Stochastic Analysis and Applications, 37(3), 361-376.

\bibitem{dip}  Di Persio, L. (2014). Backward stochastic Volterra integral equation approach to stochastic differential
utility, Int. Electron. J. Pure Appl. Math., 8, 11-15.

\bibitem{DH} Djehiche, B., \& Huang, M. (2016). A characterization of sub-game perfect equilibria for SDEs of mean-field type. Dynamic Games and Applications, 6(1), 55-81.

\bibitem{DJ1} Djordjevi{\'c}, J.  \& S. Jankovi{\'c}, (2013). On a class of backward stochastic Volterra integral equations,
Appl. Math. Lett., 26, 1192-1197.

\bibitem{DJ2}  Djordjevi{\'c}, J.  \& S. Jankovi{\'c}, (2015). Backward stochastic Volterra integral equations with additive
perturbations, Appl. Math. Comput., 265, 903-910.

\bibitem{elkaroui} El Karoui, N., Kapoudjian, C., Pardoux, {\'E}., Peng, S., \& Quenez, M. C. (1997). 
Reflected solutions of backward SDE's, and related obstacle problems for PDE's. 
The Annals of Probability, 25(2), 702-737.

\bibitem {HO} Hu, Y., \& \O ksendal, B. (2019). Linear backward stochastic
Volterra equations. Stochastic Processes and their Applications, 129(2), 626-633.

\bibitem{KO} Kromer, E. \& L. Overbeck, L. (2017).  Differentiability of BSVIEs and dynamical capital allocations,
Int. J. Theor. Appl. Finance, 20(07), 1750047.

\bibitem{lin} Lin, J. (2002). Adapted solution of backward stochastic nonlinear Volterra integral equation, Stoch. Anal. Appl. 20, 165-183.

\bibitem{peng} Peng, S. (1999). Monotonic limit theorem of BSDE and nonlinear decomposition theorem of Doob-Meyers type. Probability theory and related fields, 113(4), 473-499.
 
\bibitem{P} Popier, A. (2020). Backward stochastic Volterra integral equations with jumps in a general filtration. Preprint, arXiv:2002.06992.

\bibitem{RY} Ruvuz, D. and Yor, M. (2013). Continuous martingales and Brownian motion (Vol. 293). 
Springer Science \& Business Media.

\bibitem{swy} Shi, Y., Wang, T.  \& Yong, J. (2015). Optimal control problems of forward-backward stochastic Volterra
integral equations, Math. Control Relat. Fields, 5(3), 613-649.

\bibitem {wy} Wang, T., \& Yong, J. (2015). Comparison theorems for some
backward stochastic Volterra integral equations. Stochastic Processes and
their Applications, 125(5), 1756-1798.

\bibitem{wsy} Wang, H., Sun, J.  \&  Yong, J. (2019).  Recursive utility processes, dynamic risk measures and
quadratic backward stochastic Volterra integral equations, Appl. Math. Optim.,
https://doi.org/10.1007/s00245-019-09641-7.

\bibitem {WZ} Wang, Z., \& Zhang, X. (2007). Non-Lipschitz backward stochastic
Volterra type equations with jumps. Stochastics and Dynamics, 7(04), 479-496.

\bibitem{wyz} Wang, H., Yong, J. \& Zhang, J. (2020). Path Dependent Feynman-Kac Formula for Forward Backward Stochastic Volterra Integral Equations, Preprint, arXiv:2004.05825.

\bibitem {Yong1} Yong, J. (2006). Backward stochastic Volterra integral
equations and some related problems. Stochastic Processes and their
Applications, 116(5), 779-795.

\bibitem{Yong2}  Yong, J. (2007). Continuous-time dynamic risk measures by backward stochastic Volterra integral equa-
tions, Appl. Anal., 86, 1429-1442. 

\bibitem {Yong3} Yong, J. (2008). Well-posedness and regularity of backward
stochastic Volterra integral equations. Probability Theory and Related Fields,
142(1-2), 21-77.

\bibitem {Yong4} Yong, J. (2013). Backward stochastic Volterra integral
equations- a brief survey. Applied Mathematics-A Journal of Chinese
Universities, 28(4), 383-394.

\bibitem{yong5} Yong, J. (2012). Time-inconsistent optimal control problems and the equilibrium HJB equation, Math.
Control Relat. Fields, 2, 271-329.

\bibitem{wyy} Wei, Q., Yong, J., \& Yu, Z. (2017). Time-inconsistent recursive stochastic optimal control problems. SIAM J. Control Optim., 55(6), 4156-4201.


\bibitem{zhang} Zhang, J. (2017). {\it Backward stochastic differential equations}. Springer, New York.
\end{thebibliography}
\end{document}